\documentclass[reqno, 11pt, a4paper]{amsart}

\usepackage{amsfonts, amsthm, amsmath, amssymb}
\usepackage[hidelinks]{hyperref}
\hypersetup{colorlinks=false}
\usepackage{booktabs}
\usepackage[font=small, margin=40pt,figureposition=below]{caption}
\usepackage{graphicx}
\usepackage{tikz-cd}
\usepackage{enumitem}
\usepackage[T1]{fontenc}
\usepackage[utf8]{inputenc}
\usepackage[english]{babel}
\usepackage{amstext}
\usepackage{url}
\usepackage{cite}
\usepackage{mathtools}
\usepackage[margin=3.5cm]{geometry}

\newtheorem{theorem}{Theorem}[section]
\newtheorem{lemma}[theorem]{Lemma}
\newtheorem{corollary}[theorem]{Corollary}
\newtheorem{proposition}[theorem]{Proposition}

\theoremstyle{definition}
\newtheorem{definition}[theorem]{Definition}
\newtheorem{example}[theorem]{Example}

\theoremstyle{remark}
\newtheorem{remark}[theorem]{Remark}

\DeclareMathOperator{\charac}{char}
\DeclareMathOperator{\Z}{\mathbb{Z}}
\DeclareMathOperator{\N}{\mathbb{N}}
\DeclareMathOperator{\Q}{\mathbb{Q}}

\DeclareMathOperator{\Spec}{Spec}

\newcommand\blfootnote[1]{%
	\begingroup
	\renewcommand\thefootnote{}\footnote{#1}%
	\addtocounter{footnote}{-1}%
	\endgroup
}

\begin{document}
	\title[Common valuations of division polynomials]{Common valuations of division polynomials}
	
	\author{Bartosz Naskręcki, Matteo Verzobio}
	
	\begin{abstract}
		In this note, we prove a formula for the cancellation exponent $k_{v,n}$ between division polynomials $\psi_n$ and $\phi_n$ associated with a sequence $\{nP\}_{n\in\mathbb{N}}$ of points on an elliptic curve $E$ defined over a discrete valuation field $K$. The formula greatly generalizes the previously known special cases
		and treats also the case of non-standard Kodaira types for
		non-perfect residue fields.
	\end{abstract}
	\maketitle
	\section{Introduction}
	Let $P$ be a point on an elliptic curve $E$ defined over a discrete valuation field $K$. In this paper, we study some properties of the sequence $\{nP\}_{n\in\mathbb{N}}$ of the multiples of the point $P$. It is known that this sequence has some remarkable properties, in particular when a Weierstrass model is fixed
	\[y^2+a_1xy+a_3y=x^3+a_2x^2+a_4x+a_6\]
	the elements $nP=\left(\frac{\phi_n(P)}{\psi_n^2(P)},\frac{\omega_n(P)}{\psi_n^3(P)}\right)$ for $P\in E(K)$ satisfy certain recurrent formulas (see properties (A), (B), (C), p. 4 or \cite[Exercise 3.7]{aoc}).
	
	When the curve $E$ is defined over a number field the sequence of polynomials $(\psi_{n}(x,y))_{n\in\mathbb{N}}$ has been studied as an example of a divisibility sequence, a notion which was studied classically for Lucas sequences and by Ward for higher degree recurrences. Many results are known, including \cite{ward} and \cite{silvermanprimitive}.
	
	Over function fields, in particular over the function field of a smooth projective curve $C$ over an algebraically closed field $k$, one can study the set of effective divisors $\{D_{nP}\}$ where $D_{nP}$ is the divisor of poles of the element $x(nP)$ for a fixed $n$, cf. \cite{ingrammahesilvstreng}, \cite{Naskrecki_NYJM}, \cite{Naskrecki_Streng}, \cite{Cornelissen}.

	The question which we investigate in this paper is to what extent, for a fixed discrete valuation $v$ in the field $K$, the numbers $v(\psi_n(P))$ and $v(\phi_{n}(P))$ can be both positive and what is the maximum exponent $r$ for the power $\pi^{r}$ of the local at $v$ uniformizing element $\pi\in K$ which divides both $\psi_n(P)$ and $\phi_n(P)$. Such a situation only happens when the point $P$ has bad reduction at the place $v$ and we produce a complete formula that deals with all Kodaira reduction types. 
	
	Our main result shows that the main theorem of Yabuta and Voutier
	\cite{yv21}, which applies for finite extensions of $\Q_p$, holds for more general discrete valuation fields. Our proof applied to the case of
	$\Q_p$ is much shorter and depends exclusively on the elementary
	properties of the Néron local heights. Up to some standard facts from \cite{Lang_dioph_anal}, \cite{Lang_Fundamentals}, and \cite{Silverman_Advanced} our proof of Equations \eqref{eq:kvnnon} and \eqref{eq:kvsin} is essentially self-contained. The calculations that reveal the exact form of $k_{v,n}$ for a given reduction type of place $v$ are calculated using explicit arithmetic models, cf. \cite{cremonaprickett}, \cite{Tate_algorithm}, and \cite{Szydlo_PhD}.
	
	Let $n$ be a positive integer and define
	\begin{equation}
		k_{v,n}=\min\left\{v\left(\psi_n^2(P)\right),v\left(\phi_n(P)\right)\right\}.
	\end{equation}

	\begin{theorem}\label{main}
		Let $R$ be a discrete valuation ring with quotient field $K$. Let $v$ be the valuation of $K$. Let $E/K$ be an elliptic curve defined by a Weierstrass model in minimal form with respect to $v$, let $P=(x,y)\in E(K)$ and let $n$ be a positive integer.
		
		If $P\neq O$ is non-singular modulo $v$, then
		\begin{equation}\label{eq:kvnnon}
			k_{v,n}=\min\left(0,n^2v(x(P))\right).
		\end{equation}
		
		If $P\neq O$ is singular modulo $v$, then
		\begin{equation}\label{eq:kvsin}
			k_{v,n}=n^2c_v(P)-c_v(nP).
		\end{equation}
		
		The function $c_v(nP)$ will be defined in Definition \ref{def:cv} and it is periodic in $n$.
		The value of $n^2c_v(P)-c_v(nP)$ is computed in Table \ref{tab:common_vals}, for the case of standard reduction types at $v$, and in Table \ref{tab:common_vals_non_std} for the case of non-standard reduction types. The constants $m_P$ and $a_P$ will be defined in the next section.
	\end{theorem}
	
	\begin{table}[!htb]
		\begin{center}
			\scalebox{0.9}{\begin{tabular}{c|c|cl}
					Kodaira symbol       & $m_{P}$ & $k_{v,n}=n^2c_v(P)-c_v(nP)$ & \\ 
					\hline
					$III^{*}$                     &   $2$   &     $3n^2/2$       & \text{{\rm if} $n \equiv 0 \bmod{m_{P}}$}     \\
					&                  & $3(n^2-1)/2$     & \text{{\rm if} $n \not\equiv 0 \bmod{m_{P}}$} \\ \hline
					$IV^{*}$                      &   $3$   &     $4n^2/3$     & \text{{\rm if} $n \equiv 0 \bmod{m_{P}}$}     \\
					&         &         $4(n^2-1)/3$    & \text{{\rm if} $n \not\equiv 0 \bmod{m_{P}}$} \\ \hline
					$III$                         &   $2$   &     $n^2/2$      & \text{{\rm if} $n \equiv 0 \bmod{m_{P}}$}     \\
					&         &         $(n^2-1)/2$     & \text{{\rm if} $n \not\equiv 0 \bmod{m_{P}}$} \\ \hline
					$IV$                          &   $3$   &     $2n^2/3$     & \text{{\rm if} $n \equiv 0 \bmod{m_{P}}$}     \\
					&         &          $2(n^2-1)/3$     & \text{{\rm if} $n \not\equiv 0 \bmod{m_{P}}$} \\ \hline
					$I_{m}^{*}$       &   $2$   &    $n^2$      & \text{{\rm if} $n \equiv 0 \bmod{m_{P}}$}     \\
					&         &            $n^2-1$       & \text{{\rm if} $n \not\equiv 0 \bmod{m_{P}}$} \\ \hline
					$I_{m}^{*}$
					&   $4$   & $n^2(m+4)/4$        & \text{{\rm if} $n \equiv 0 \bmod{m_{P}}$}   \\
					
					&         &            $(n^2-1)(m+4)/4$ & \text{{\rm if} $n \equiv 1,3 \bmod{m_{P}}$} \\
					&         &            $(n^2(m+4)/4)-1$       & \text{{\rm if} $n \equiv 2 \bmod{m_{P}}$}   \\  \hline
					$I_{m}$                       & $\dfrac{m}{\gcd \left( a_{P}, m \right)}$
					& $n^2\dfrac{a_{P}\left( m-a_{P} \right)}{m}-\dfrac{n'\left( m-n' \right)}{m}$
					& $n':= a_P n (\textrm{mod }m)$ \\ 
			\end{tabular}}    
		\end{center}
		\caption{Tables of the common valuation $k_{v,n}$ for standard Kodaira types.}\label{tab:common_vals}
	\end{table}

	\begin{table}[!htb]
		\begin{center}
			\scalebox{0.9}{\begin{tabular}{c|c|cl}
					Kodaira symbol       & $m_{P}$ & $k_{v,n}=n^2c_v(P)-c_v(nP)$ & \\ 
					\hline
					$X_2$                     &   2   &     $n^2$       & \text{{\rm if} $n \equiv 0 \bmod{m_{P}}$}     \\
					&                  & $n^2-1$     & \text{{\rm if} $n \not\equiv 0 \bmod{m_{P}}$} \\ \hline
					$K_{2m}$                     &   2   &     $\frac{mn^2}{2}$       & \text{{\rm if} $n \equiv 0 \bmod{m_{P}}$}     \\
					&                  & $\frac{m(n^2-1)}{2}$     & \text{{\rm if} $n \not\equiv 0 \bmod{m_{P}}$} \\ \hline
					$T_{m}$                     &   2   &     $n^2$       & \text{{\rm if} $n \equiv 0 \bmod{m_{P}}$}     \\
					&                  & $n^2-1$     & \text{{\rm if} $n \not\equiv 0 \bmod{m_{P}}$} \\ \hline
			\end{tabular}}    
		\end{center}
		\caption{Tables of the common valuation $k_{v,n}$ for non-standard Kodaira types.}\label{tab:common_vals_non_std}
	\end{table}
	
	If the residue field $\kappa$ is perfect, then the Kodaira symbol of the curve is one of those of Table \ref{tab:common_vals} and it can be computed using \cite[Table C.15.1]{aoc}. If $\kappa$ is not perfect, one can compute the Kodaira symbol of the curve using \cite{Szydlo_PhD,szydlo}. If the Kodaira symbol is not one of the symbols in Table \ref{tab:common_vals} and $P$ is singular modulo $v$, then the Kodaira symbol is one of those in Table \ref{tab:common_vals_non_std}. 
	
	We also emphasize that the formula for the cancellation $k_{v,n}$ is a quadratic polynomial in $n$ with rational coefficients which depend only on the correction terms $c_v(P)$. When $K$ is a function field of a curve over an algebraically closed field, the quantity $c_v(P)$ was defined by other means in Cox-Zucker \cite{coxzucker} and used extensively by \cite{shioda} in his theory of Mordell-Weil lattices. 
	There are two key facts that make the proof ultimately very short.
	
	\begin{itemize}
		\item[1.] The N\'{e}ron local height is a sum of two ingredients: local intersection pairing with the zero section and the ``correction part'' which depends only on the component which the given point hits at the place $v$, cf. \cite{Lang_dioph_anal}, \cite{J_S_Muller},\cite{Busch_Muller} \cite{Holmes_PhD}.
		\item[2.] The valuation $v(\psi_n(P))$ is a quadratic polynomial in $n$ whose coefficients depend on the local N\'{e}ron height at arguments $nP$ and $P$, respectively, and on the valuation of the minimal discriminant $\Delta$ of the elliptic curve $E$, cf. Lemmas \ref{lem:valuation_lem_1},\ref{lemma:psiinter}. In the case of finite extensions of $\Q_p$, this was already proven in \cite{Stange}.
	\end{itemize}
	
	On the way we reprove in our setting some results of \cite{ayad} and obtain a much simpler analysis of the ''troublemaker sequences'' used by Stange \cite{Stange} and \cite{yv21} in their proofs. 
	In the number field case, the sequence $k_{v,n}$ received considerable attention during the last 30 years, with \cite{ayad}, \cite{chhadiv}, \cite{ingram} attempting to
	understand the shape of $k_{v,n}$. In this case, \cite{yv21} established the precise value
	of $k_{v,n}$. The goal of this paper is to extend their result to the most general setting. One can easily check that our Table \ref{tab:common_vals} agrees with \cite[Table 1.1]{yv21}.
	
	The study of the sequence $k_{v,n}$ has some interesting applications. For example, knowing the behaviour of $k_{v,n}$ was necessary to show that every elliptic divisibility sequence satisfies a recurrence relation in the rational case, cf. \cite{recurrence}. In Section \ref{sec:exam} we will show how Theorem \ref{main} helps to generalize that result. Another application can be found in \cite{Stange} and \cite{ingram}, where the study of $k_{v,n}$ was necessary to study the problem of understanding when a multiple of $P$ is integral.
	
	\section*{Acknowledgements}
	We want to thank Stefan Barańczuk, Matija Kazalicki, Joseph Silverman, and Paul Voutier for the comments on the earlier version of this paper. 
	We thank the anonymous referee for their careful reading of our manuscript and their many insightful comments and suggestions.
	
	The first author acknowledges the support by Dioscuri program initiated by the Max Planck Society, jointly managed with the National Science Centre (Poland), and mutually funded by the Polish Ministry of Science and Higher Education and the German Federal Ministry of Education and Research.
	The second author has been supported by MIUR (Italy) through PRIN 2017 ``Geometric, algebraic and analytic methods in arithmetic'' and has received funding from the European Union’s Horizon 2020 research and 
	innovation program under the Marie Skłodowska-Curie Grant Agreement No. 101034413.
	
	\section{First definitions}\label{sec:firstdef}
	Let $R$ be a discrete valuation ring with quotient field $K$ and residue field $\kappa$.
	Let $v$ be the valuation of $K$ and assume that $v(K^*)=\Z$. Let $E$ be an elliptic curve over $K$ which is a generic fibre of the regular proper model $\pi:\mathcal{E}\rightarrow\Spec R=\{o,s\}$. The fibre $E_s=\pi^{-1}(s)$ is the special fibre (singular or not). We denote by $\widetilde{\mathcal{E}}\rightarrow\Spec R=\{o,s\}$ the associated N\'{e}ron model which is the subscheme of smooth points over $R$. 
	Assume that Weierstrass model of $E$ is in minimal form with respect to $v$ and let $P=(x,y)\in E(K)$. We denote by $\Delta$ the discriminant of the minimal model $E$.
	
	The point $P\in E(K)$ extends to a section $\sigma_{P}:\Spec R\rightarrow\overline{\mathcal{E}}$ by the N\'{e}ron model property. Let $\Phi_{v}$ denote the group of components of the special fibre $E_s$. We have a natural homomorphism $comp_{v}:E(K)\rightarrow\Phi_{v}$ which sends the point $P$ to the element of the component group. We say that the point $P$ is non-singular modulo $v$ if the image $comp_v(P)$ is the identity element, otherwise the point $P$ is singular modulo $v$. 
	
	Let $E_0(K)\subset E(K)$ be the kernel of $comp_{v}$, i.e. the group of points that reduce to non-singular points modulo $v$. Every element $comp_{v}(P)$ belongs to a cyclic component of $\Phi_v$. Let $a_P$ be the order of this component and let $m_P$ be the order of $comp_v(P)$ in $\Phi_v$. 
	
	In the proof of Theorem \ref{main}, without loss of generality, $K$ may be replaced by its completion. So, we will assume that $K$ is complete (with respect to $v$).
	
	\begin{definition}
		Let $(P.O)_v$ denote the local intersection number of the point $P\in E(K)\setminus\{O\}$ with the zero point $O$ at $v$ defined by the formula
		\[(P.O)_v=\max\{0,-v(x(P))\}\]
		for the minimal model $E$. In particular, $(P.O)_v>0$ if and only if $P$ reduces to the identity modulo $v$ and then $P$ is non-singular. For more details on local intersection, see \cite[Section IV.7]{Silverman_Advanced}.
	\end{definition}
	\begin{remark}
		For the behaviour of the sequence of integers $\{(nP.O)_v\}_n$, see \cite[Section 7 and Lemma 8.2]{Naskrecki_NYJM} in the function field case and \cite[Lemma 5.1]{Stange} in the number field case.
	\end{remark}
	\begin{definition}\label{def:cv}
		
		With a pair $(P,v)$ we associate a rational number $c_v(P)$ which is called \textit{the correction term}.  Let  $\widehat{\lambda_v}(P)$ be the local N\'{e}ron height as defined in \cite[Theorem III.4.1]{Lang_dioph_anal}. Define
		\begin{equation}\label{eq:cvP}
			c_v(P)=(P.O)_v+v(\Delta)/6-2\widehat{\lambda_v}(P).
		\end{equation}
	\end{definition}
	
	\begin{remark}
		Observe that $c_v(P)$ depends only on $comp_v(P)$. This follows from \cite[Theorem 11.5.1]{Lang_Fundamentals} when $P$ is singular modulo $v$ and from \cite[Theorem VI.4.1]{Silverman_Advanced} when $P$ is non-singular modulo $v$. In particular, \cite[Theorem VI.4.1]{Silverman_Advanced} shows that $c_v(P)=0$ when $P$ is non-singular modulo $v$.
		In Lemma \ref{tab:corr_terms} we will show how to explicitly compute $c_v(P)$.
		Finally, note that the definition of local N\'{e}ron height in \cite[Theorem III.4.1]{Lang_dioph_anal} is given for discrete valuation fields, as is noted by Lang in \cite[p. 66]{Lang_dioph_anal}.
	\end{remark}
	
	The division polynomials are defined in \cite[Exercise 3.7]{aoc}. We recall the properties that we will use:
	\begin{itemize}
		\item[(A)] If $nP$ is not equal to the identity $O$ of the curve, then \[x(nP)=\frac{\phi_n(P)}{\psi_n^2(P)};\]
		\item[(B)] The polynomials have integer coefficients and depends only on $x(P)$. Seen as polynomials in $x(P)$, $\psi_n^2(x)$ has degree $n^2-1$ and $\phi_n(x)$ is monic and has degree $n^2$;
		\item[(C)] For all $n\geq 1$, \[\phi_n(P)=x(P)\psi_n^2(P)-\psi_{n-1}(P)\psi_{n+1}(P).\]
		\item[(D)]For all $n\geq m\geq r$, \[
		\psi_{n+m}(P)\psi_{n-m}(P)\psi_r^2(P)=\psi_{n+r}(P)\psi_{n-r}(P)\psi_m^2(P)-\psi_{m+r}(P)\psi_{m-r}(P)\psi_n^2.
		\]
	\end{itemize}
	
	\begin{remark}
		Our proof clearly works also for number fields. In that setting the result was known \cite{yv21}, but the two proofs are completely different. Indeed, the proof of Yabuta and Voutier is almost completely arithmetical, while ours is more geometrical.
		
	\end{remark}
	\begin{remark}
		We will work assuming that $v(0)=\infty$. Under this assumption, our theorem works also in the case when $nP=O$. In this case $\psi_n^2(P)=0$ and then $k_{v,n}(P)=v(\phi_n(P))$.
	\end{remark}
	
	\section{Proof of the main theorem}
	We start this section by recalling some classical facts on the local N\'{e}ron height.
	\begin{lemma}\label{lemma:properties}
		The following hold:
		\begin{itemize}
			\item For all $R,Q\in E(K)$ with $R,Q,R\pm Q\neq O$,
			\begin{equation}\label{eq:lambda}
				\widehat{\lambda_v}(R+Q)+\widehat{\lambda_v}(R-Q)=2\widehat{\lambda_v}(R)+2\widehat{\lambda_v}(Q)+v(x(R)-x(Q))-v(\Delta)/6.
			\end{equation}
			\item The N\'{e}ron local height does not change if we replace $K$ with a finite extension. Moreover, $\widehat{\lambda_v}(\cdot)$ does not depend on the choice of the minimal Weierstrass model defining the curve.
			\item If $comp_v(P)=comp_v(P')$, then $\widehat{\lambda_v}(P)=\widehat{\lambda_v}(P')$.
			\item If $P\notin E_0(K)$, then $\widehat{\lambda_v}(P)$ just depends on the image of $P$ in $E(K)/E_0(K)$.
		\end{itemize}
	\end{lemma}
	\begin{proof}
		For the first two statements, see \cite[Theorem III.4.1 and Page 63]{Lang_dioph_anal}. For the last two, see \cite[Theorem 11.5.1]{Lang_Fundamentals}.
	\end{proof}
	\begin{lemma}\label{lemma:advsilv}
		If $P$ has non-singular reduction, then \[\widehat{\lambda_v}(P)=v(\Delta)/12+(P.O)_v/2.\]
	\end{lemma}
	\begin{proof}
		See \cite[Theorem VI.4.1]{Silverman_Advanced} for the case of perfect fields and \cite[Chapter III, \S 4, \S 5]{Lang_dioph_anal} for the general case. Observe that $\widehat{\lambda_v}(P)$ is denote by $\lambda_v(P)$ in both \cite[Chapter III, \S 4, \S 5]{Lang_dioph_anal} and \cite{Silverman_Advanced}.
	\end{proof}
	\begin{lemma}\label{lem:valuation_lem_1}
		If $nP\neq O$, then 
		\begin{equation}\label{eq:psi_lambda_formula}
			\widehat{\lambda_v}(nP) = n^2 \widehat{\lambda_v}(P)+v(\psi_{n}(P))-\frac{n^2-1}{12}v(\Delta),
		\end{equation}
		where $\Delta$ is the discriminant of the curve.
	\end{lemma}
	\begin{proof}
		This lemma is stated as an exercise in \cite[Exercise 6.4 (e)]{Silverman_Advanced}. First, assume that $P$ is a non-torsion point.
		We prove the lemma by induction. If $n=1$, it is obvious since $\psi_1(P)=1$. If $n=2$, then it follows from \cite[Theorem VI.1.1]{Silverman_Advanced}. Note that, in \cite{Silverman_Advanced}, it is assumed that $\kappa$ is perfect, but the proof of the theorem works in the exact same way without that requirement, cf. \cite[Chapter III, \S 4, \S 5]{Lang_dioph_anal}. So, we assume that the lemma is true for $i\leq n$ and we show that it holds also for $n+1\geq 3$. Put $R=nP$ and $Q=P$. So, $R,Q,R\pm Q\neq O$. Then, by Lemma \ref{lemma:properties},
		\begin{equation}\label{eq:n+1}
			\widehat{\lambda_v}((n+1)P)=-\widehat{\lambda_v}((n-1)P)+2\widehat{\lambda_v}(nP)+2\widehat{\lambda_v}(P)+v(x(nP)-x(P))-\frac{v(\Delta)}6.
		\end{equation}
		For the definition of division polynomials we have
		\begin{align}\label{eq:lam}
			x(nP)-x(P)&=\frac{\phi_n(P)}{\psi_n^2(P)}-x(P)\\&=x(P)+\frac{\psi_{n+1}(P)\psi_{n-1}(P)}{\psi_n^2(P)}-x(P)\nonumber \\&=\frac{\psi_{n+1}(P)\psi_{n-1}(P)}{\psi_n^2(P)}.\nonumber
		\end{align}
		Put $d_n=\widehat{\lambda_v}(nP)-v(\psi_{n}(P))$ and then combining (\ref{eq:n+1}) and (\ref{eq:lam}) we have
		\[
		d_{n+1}=2d_n-d_{n-1}+2d_1-\frac{v(\Delta)}6.
		\]
		By induction, we have $d_i=i^2d_1-(i^2-1)v(\Delta)/12$ for $i\leq n$ and then
		\begin{align*}
			d_{n+1}&=(2n^2-(n-1)^2+2)d_1-v(\Delta)\left(\frac{2(n^2-1)-((n-1)^2-1)+2}{12}\right)\\&=(n+1)^2d_1-v(\Delta)\left(\frac{(n+1)^2-1}{12}\right).
		\end{align*}
		Therefore,
		\begin{align*}
			\widehat{\lambda_v}((n+1)P)&=v(\psi_{n+1}(P))+d_{n+1}\\&=v(\psi_{n+1}(P))+(n+1)^2\widehat{\lambda_v}(P)-v(\Delta)\left(\frac{(n+1)^2-1}{12}\right).
		\end{align*}
		
		Now, we need to prove the lemma in the case when $P$ is a torsion point. Fix $n\in \mathbb{N}$. First, we can assume that $K$ is complete since clearly the statement does not change. In a finite extension of $K$ we can find a non-torsion point $P'$ such that $(P-P'.O)_v$ is very large. The choice of $P'$ will depend on $n$. Thanks to Lemma \ref{lemma:properties}, the N\'{e}ron local height does not change if we replace $K$ with a finite extension, so we can assume $P'\in E(K)$. Since $(P-P'.O)_v$ is very large, $(nP.O)_v=(nP'.O)_v$ and $(P.O)_v=(P'.O)_v$ (this is certainly true if $(P-P'.O)_v>(nP.O)_v$). Moreover, $comp_v(P)=comp_v(P')$ and then, from Lemma \ref{lemma:properties}, $\widehat{\lambda_v}(P)=\widehat{\lambda_v}(P')$ and $\widehat{\lambda_v}(nP)=\widehat{\lambda_v}(nP')$. Furthermore, $v(\psi_n(P))=v(\psi_n(P'))$ since $\psi_n$ is a continuous function and $v$ is discrete. We know that the lemma holds for $P'$ since it is a non-torsion point and then
		\begin{align*}
			\widehat{\lambda_v}(nP)=\widehat{\lambda_v}(nP') &= n^2 \widehat{\lambda_v}(P')+v(\psi_{n}(P'))-\frac{n^2-1}{12}v(\Delta)\\&=n^2 \widehat{\lambda_v}(P)+v(\psi_{n}(P))-\frac{n^2-1}{12}v(\Delta).
		\end{align*}
	\end{proof}
	
	Recall that, in the case $\kappa$ perfect, the Kodaira symbol of an elliptic curve can be computed using \cite[Table C.15.1]{aoc}. The analogous reference for the non-standard Kodaira symbols is \cite{Szydlo_PhD,szydlo}. Recall that $a_P$ and $m_P$ are defined at the beginning of Section \ref{sec:firstdef}.
	\begin{lemma}\label{tab:corr_terms}
		The value of $c_v(P)$ is provided as a function of the Kodaira symbol and $m_P$ in the following table.
		\begin{center}
			\begin{tabular}{c|c|c}
				Kodaira symbol       & $m_{P}$& $c_v(P)$ \\ \hline
				Any                   &   $1$   &     $0$       \\ \hline
				$III$                         &   $2$   &     $1/2$          \\
				\hline
				
				$III^{*}$                     &   $2$   &     $3/2$        \\
				\hline
				$I_{m}^{*}$       &   $2$   &    $1$       \\
				\hline
				$X_2$ & $2$ & $1$\\
				\hline
				$K_{2m}$ & $2$ & $m/2$\\
				\hline
				$T_{m}$ & $2$ & $1$\\
				\hline
				$IV^{*}$                      &   $3$   &     $4/3$         \\
				\hline
				$IV$                          &   $3$   &     $2/3$      \\
				\hline
				$I_{m}^{*}$
				&   $4$   & $(m+4)/4$            \\
				\hline
				$I_{m}$                       & $\dfrac{m}{\gcd \left( a_{P}, m \right)}$
				& $\dfrac{a_{P}\left( m-a_{P} \right)}{m}$\\ 
			\end{tabular}
		\end{center}
		In the last line, we are assuming that $0<a_P< m$.
	\end{lemma}
	\begin{proof}
		Recall that, by Definition \ref{def:cv},
		\[c_v(P)=(P.O)_v+v(\Delta)/6-2\widehat{\lambda_v}(P).\]
		Assume that $P$ has non-singular reduction. By Lemma \ref{lemma:advsilv}, we have $\widehat{\lambda_v}(P)=v(\Delta)/12+(P.O)_v/2$. Hence, $c_v(P)=0$. 
		
		If $P$ has singular reduction, then $(P.O)_v=0$ since the identity is a non-singular point. 
		The calculation of the values of $c_v(P)$ is based on the information obtained from the resolution of singularities performed on Weierstrass minimal model. By Lemma \ref{lemma:properties}, $\widehat{\lambda_v}(P)$ does not depend on the choice of the minimal Weierstrass model defining the curve. So, $c_v(P)$ does not depend on the choice of the minimal Weierstrass model. 
		
		If the residue field $\kappa$ is perfect, Tate algorithm \cite{Tate_algorithm} proves the existence of the flat proper regular model of an elliptic curve. 
		The procedure of Tate was extended in the work of Michael Szydlo \cite{Szydlo_PhD,szydlo} for any elliptic curve over DVR, in particular to the case when $\kappa$ is not a perfect field. We will refer to this general procedure as GTA (Generalized Tate's algorithm).
		
		When $\charac\kappa\neq 2,3$ the values of $c_v(P)$ can be determined from the data obtained by the procedure of Tate. In fact, one follows the calculations in \cite[Proposition 6]{cremonaprickett}
		and notes that the valuation of the minimal model and the valuation of the $(x,y)$ coordinates of the point which reduces to the singularity are uniform across all fields with $\charac\kappa\neq 2,3$, cf. \cite[\S 4.1]{szydlo} and in particular \cite[Table 1]{szydlo}. 
		
		We will explain now in detail the computations for residue fields of characteristic $2$ and $3$, when new reduction types arise and some care must be taken. We essentially follow the strategy laid out above and supplement it with the data about valuations of the coefficients provided in \cite{Szydlo_PhD,szydlo}. When the fiber at $v$ has reduction type $I_0$, $II$, $II^*$ or any of the following non-standard Kodaira types: $Z_1, Z_2, X_1, Y_1, Y_2,Y_3, K_{2n}', K_{2n+1}$, the group of components is trivial (see  \cite[Equation (20)]{szydlo}). 
		
		\textbf{Case $m_{P}=1$.} In such a case $P$ has non-singular reduction and then, as we pointed out at the beginning of the proof, $c_{v}(P) = 0$.
		
		\textbf{Case $m_{P} = 2$.} Since $comp_{v}$ is a group homomorphism and $\Phi_{v}\cong \mathbb{Z}/2\mathbb{Z}$ it follows that if $P\not\in E_{0}(K_{v})$, then $c_{v}(2P) = 0$ and $c_{v}(P) = c_{v}(3P)$. Moreover, $(P.O)_{v} = (3P.O)_{v} = 0$, hence $\widehat{\lambda_v}(P) = \widehat{\lambda_v}(3P)$. We get from Lemma \ref{lem:valuation_lem_1} for $n=3$ that
		\[8\widehat{\lambda_v}(P) = 2v(\Delta)/3 - v(\psi_{3}(P)).\]
		From the definition of $c_{v}(P)$ (see Equation \eqref{def:cv}) we get that
		\[c_{v}(P) = v(\psi_{3}(P))/4.\]
		For notational convenience, put $x=x(P)$.
		Since $\psi_{3}(P) = 3x^4+b_2 x^3+3b_4x^2+3b_6x+b_8$ it remains to calculate the valuation estimates for the suitable quantities. For their definition, see \cite[Section III.1]{aoc}. By \cite[Equation (20)]{szydlo}, the only non-standard Kodaira types that we have to consider are $X_2$, $K_{2m}$, and $T_m$. 
		
		Types $III, III^*, I_{m}^*$. We follow the argument in \cite[Proposition 6]{cremonaprickett} combined with the valuation tables in \cite{szydlo} for the non-perfect field cases. Notice that, for these types, Tate's algorithm works in the exact same way in the perfect and non-perfect field case (see in particular \cite[Tables 1, 4, 5]{szydlo}).
		
		Type $X_2$. It follows from \cite[p.41]{Szydlo_PhD} that we can construct a minimal model where $v(a_1)\geq 1$, $v(a_2)\geq 2$, $v(a_3)\geq 2$, $v(a_4)=2$, $v(a_6)\geq 4$ and the sequence of blowups at the singular points resolve the model in such a way that the point $P=(x,y)\not\in E_0(K_v)$ has $v(x)\geq 2$, cf.\cite[Table 5]{szydlo}. It follows that $v(b_2)\geq 2$, $v(b_4)\geq 3$, $v(b_6)\geq 4$ and finally $v(b_8) = 4$. Thus, we have $v(\psi_3(P))=4$ and then $c_v(P) = 1$.
		
		Type $K_{2m}$. It follows from \cite[Table 6]{szydlo} that $v(b_2)\geq 2$, $v(b_4)\geq 2$, $v(b_6)\geq 2m$, and $v(b_8) = 2m$. The resolution procedure in \cite[\S 6.12]{Szydlo_PhD} implies that for the minimal model with the previous assumptions we have $v(x)>m$. Hence, we obtain $v(\psi_{3}(P)) = 2m$, proving that $c_{v}(P) = m/2$.
		
		Type $T_{m}$. In the resolution \cite[pp. 48-51]{Szydlo_PhD} the valuation of $x$ is equal to $1$, cf. \cite[p. 51]{Szydlo_PhD}. It follows from \cite[Table 7]{szydlo} that all terms of $\psi_{3}(P)$ but $3x^4$ have valuation at least $5$. Hence $v(\psi_{3}(P)) = 4$ and thus $c_v(P) = 1$.
		
		\textbf{Case $m_{P}=3$.} Type $IV, IV^*$.
		First of all, observe that $\widehat{\lambda_v}(P)=\widehat{\lambda_v}(-P)$. Indeed, if we put $\widehat{\lambda_v'}(P)\coloneqq\widehat{\lambda_v}(-P)$, then $\widehat{\lambda_v'}$ satisfies (\ref{eq:lambda}). Thanks to \cite[Theorem III.4.1]{Lang_dioph_anal}, we have $\widehat{\lambda_v'}=\widehat{\lambda_v}$ since $\widehat{\lambda_v}$ is unique. So, $\widehat{\lambda_v}(2P)=\widehat{\lambda_v}(-P)=\widehat{\lambda_v}(P)$.
		Using Lemma \ref{lem:valuation_lem_1} with $n=2$, we have
		\[
		3\widehat{\lambda_v}(P)=\frac{v(\Delta)}{4}-v(\psi_2(P))
		\]
		and, by \eqref{eq:cvP},
		\[
		c_v(P)=(P.O)_v+\frac{v(\Delta)}6-2\widehat{\lambda_v}(P)=\frac{v(\psi_2^2(P))}{3}.
		\]
		From the definition of $\psi_2^2(P)$ we have $\psi_2^2(P)=4x(P)^3+b_2x(P)^2+2b_4x(P)+b_6$. Since $\widehat{\lambda_v}(P)$ does not depend on the choice of the minimal Weierstrass equation defining the curve, we can assume that the coefficients satisfy GTA. In the case for the type $IV$, we have $v(b_2)>0$, $v(b_4)>1$, and $v(b_6)=2$. Moreover, $v(x(P))>0$  holds due to the resolution construction. 
		Hence,
		\[
		2=v(b_6)=v(4x(P)^3+b_2x(P)^2+2b_4x(P)+b_6)=v(\psi_2^2(P))=3c_v(P)
		\]
		and we are done. A very analogous calculation reveals that for the type $IV^*$ we have $c_{v}(P)=4/3$.
		
		\textbf{Case $m_{P}>3$.}  We obtain for every residue field, even non-perfect the values of $c_v(P)$ are subject to more complicated counting rule. We basically follow the argument described in \cite{cremonaprickett},  enhanced by the classification of fibers by Szydlo \cite{szydlo}.
		
		Type $I_m^*$ $(m_P=4)$. Following \cite[p.353]{silvermancomp} we note that if $P\not\in E_{0}(K_v)$, then $comp_v(P)=comp_v(-3P)=comp_v(3P)$ since $comp_v$ is a homomorphism and $\widehat{\lambda}_{v}$ is unique, cf. case $m_P=3$. Therefore, we can apply again Lemma \ref{lem:valuation_lem_1} with $n=3$ to conclude that $c_{v}(P) = v(\psi_{3}(P))/4$. The explicit formulas follow from the inspections of Tables in \cite{szydlo}.
		
		Type $I_m$. The correction formula for the multiplicative type can be explained in several ways. The down-to-earth approach is to invoke the theory of the Tate curve \cite[III \S 5]{Lang_dioph_anal} which is independent of the residue field. We extract $c_v(P) = k(m-k)/m$ for the intersection of $P$ with the $k$-th cyclic component from \cite[Thm. 3.5.1]{Lang_dioph_anal}, cf. \cite[Thm. 6.4.2(b)]{Silverman_Advanced} and \cite[Propositions 5,6]{cremonaprickett}. 
	\end{proof}
	\begin{remark}
		A more conceptual approach to the proof of Lemma \ref{tab:corr_terms} follows from abstract intersection theory developed by N\'{e}ron, cf. \cite{J_S_Muller}. This is an analogue of the calculations one can do in the theory of elliptic surfaces, cf. \cite[III \S 8]{Silverman_Advanced}. This is a very general approach that would in principle allow us to calculate all correction terms using just the suitable fibral divisor. In fact, none of the references treats properly the case of non-perfect fields so we refrain from giving here all the fine details of this argument.     
	\end{remark}
	\begin{lemma}\label{lemma:psiinter}
		Let $P$ be a point on the elliptic curve $E$ as above. For every $n$ such that $nP\neq O$, we have
		\[v(\psi_n^2(P)) = (nP.O)_v-n^2(P.O)_v+n^2 c_v(P)-c_v(nP).\]
		In particular, $n^2 c_v(P)-c_v(nP)$ is an integer.
	\end{lemma}
	\begin{proof}
		
		From (\ref{eq:psi_lambda_formula}) we have
		\[
		v(\psi_{n}(P))=\widehat{\lambda_v}(nP) +\frac{n^2-1}{12}v(\Delta)- n^2 \widehat{\lambda_v}(P).
		\]
		From Definition \ref{def:cv} we have
		\[
		\widehat{\lambda_v}(P)=\frac{v(\Delta)}{12}-\frac{c_v(P)}{2}+\frac{(P.O)_v}2.
		\]
		Hence,
		\begin{align*}
			v(\psi_{n}(P))&=\widehat{\lambda_v}(nP) +\frac{n^2-1}{12}v(\Delta)- n^2 \widehat{\lambda_v}(P)\\&
			=\frac{v(\Delta)}{12}-\frac{c_v(nP)}2+\frac{(nP.O)_v}2 +\frac{n^2-1}{12}v(\Delta) \\&-n^2\frac{v(\Delta)}{12}+n^2\frac{c_v(P)}2-\frac{n^2(P.O)_v}2\\&
			=\frac{(nP.O)_v}2-\frac{n^2(P.O)_v}2+\frac{1}{2}n^2 c_v(P)-\frac{1}{2}c_v(nP).
		\end{align*}
	\end{proof}
	
	\begin{remark}
		Thanks to the previous lemma, one can easily show that the valuation $v(\psi_{n}(P))$ is independent of the choice of the minimal Weierstrass model.
	\end{remark}
	
	Put $n_P$ as the smallest positive integer such that $n_P P$ reduces to the identity modulo $v$. Observe that $n_P$ does not necessarily exist.
	\begin{lemma}\label{lemma:div}
		Let $k$ be an integer. If $(kP.O)_v>0$, then $(mkP.O)_v\geq (kP.O)_v$ for all $m\geq 1$. Moreover, if $n_P$ exists, then $(kP.O)_v>0$ if and only if $n_P\mid k$.
	\end{lemma}
	\begin{proof}
		Put $P'\coloneqq kP$ and by assumption we have $v(x(P'))<0$. Hence, we have $v(\phi_m(P'))=m^2v(x(P'))$ (recall that $\phi_m(P')$, seen as a polynomial in $x(P')$, is monic) and $v(\psi_m^2(P'))\geq (m^2-1)v(x(P'))$. So, 
		\[
		v(x(mkP))=v(\phi_m(P'))-v(\psi_m^2(P'))\leq v(x(P'))=v(x(kP))<0.
		\]
		
		Now, we prove the second part of the lemma.
		If $n_P\mid k$, then $(kP.O)_v>0$. If $n_P\nmid k$, then $k=qn_P+r$ with $0<r<n_P$. Recall that the identity is a non-singular point of the curve reduced modulo $v$. Hence, if $(kP.O)_v>0$, then $kP$ reduces to the identity modulo $v$ and then also the point $rP=kP-q(n_PP)$ reduces to the identity. This is absurd from the definition of $n_P$.
	\end{proof}
	\begin{proposition}
		Assume that $P$ has non-singular reduction and $(P.O)_v=0$. Then, $k_{v,n}(P)=0$ for all $n\geq 1$.
	\end{proposition}
	\begin{remark}
		This proposition was proved for number fields by Ayad \cite{ayad}. Their proof works also in our case, but it is very different from ours.
	\end{remark}
	\begin{proof} 
		Since $P$ has non-singular reduction, then $c_v(kP)=0$ for any $k$ (see Lemma \ref{tab:corr_terms}). Hence, from Lemma \ref{lemma:psiinter}, $v(\psi_k(P)) = (kP.O)_v$ for all $k$ such that $kP\neq O$.  
		
		Assume $nP\neq O$.
		If $(nP.O)_v>0$, then due to Lemma \ref{lemma:div} the numbers $((n-1)P.O)_v$ and $ ((n+1)P.O)_v$ are equal to zero. Therefore,
		$v(\psi_n(P)^2) =(nP.O)_v>0$, $v(\psi_{n-1}(P)\psi_{n+1}(P))=0$ and $v(\phi_n(P)) = 0$ from the properties (A), (B), (C), proving $k_{v,n}(P)=0$. 
		
		If $(nP.O)_v = 0$, then $v(\psi_n(P)^2)=0$ and $v(\phi_{n}(P))\geq 0$, hence again $k_{v,n}(P)=0$. Here we are using that the polynomials $\psi_n^2$ and $\phi_n$ have integer coefficients.
		
		Assume now that $nP=O$. 
		Then $\psi_n^2(P)=0$ and so $\phi_n(P)=\psi_{n-1}(P)\psi_{n+1}(P)$. 
		Observe that $((n\pm 1)P.O)_v=0$ since otherwise $(P.O)_v$ would be strictly positive. 
		Hence, $k_{v,n\pm1}(P)=v(\psi_{n\pm1}^2(P))$. 
		For the first part of the lemma, $k_{v,n\pm1}(P)=0$ and then
		\[
		v(\phi_n(P))=v(\psi_{n-1}(P)\psi_{n+1}(P))=\frac{k_{v,n-1}(P)+k_{v,n+1}(P)}2=0.
		\]
	\end{proof}
	Now, we are ready to prove our main result.
	\begin{proof}[Proof of Theorem \ref{main}]
		Assume that $P$ has non-singular reduction.
		If $v(x(P))< 0$, then, using that $\psi_n^2$ (resp. $\phi_n$) has integer coefficients and degree $n^2-1$ (resp. $n^2$), we have $v(\psi_n^2(P))\geq (n^2-1)v(x(P))$ and $v(\phi_n(P))=n^2v(x(P))$ (recall that $\phi_n$ is monic). So, $k_{v,n}(P)=n^2v(x(P))$. If $v(x(P))\geq 0$, then $(P.O)_v=0$ and we conclude thanks to the previous proposition.
		
		Assume now that $P$ has singular reduction, that $n_P$ exists, and that $n_P\mid n$. Then, $v(x(nP))<0$. Therefore, $v(\phi_n(P))<v(\psi_n^2(P))$ and so $k_{v,n}=v(\phi_n(P))$. Since $P$ is singular we have $v(x(P))\geq 0$. Indeed, if $v(x(P))< 0$ then $P$ is the identity modulo $v$ and the identity is a non-singular point. Therefore, $v(\phi_n(P))< v(x(P)\psi_n^2(P))$. Recall that $\phi_n(P)=x(P)\psi_n^2(P)+\psi_{n-1}(P)\psi_{n+1}(P)$. So, \[v(\phi_n(P))=v(\psi_{n+ 1}(P))+v(\psi_{n- 1}(P)).\] Observe that $(n\pm 1)P\neq O$ and then in particular $((n\pm1)P.O)_v=0$. Recall that $c_v(P)$ depends only on $comp_v(P)$ and observe that $c_v(P)=c_v(-P)$. Thus, $c_v((n\pm 1)P)=c_v(P)$ since $m_P\mid n_P\mid n$. Hence, we can apply Lemma \ref{lemma:psiinter} and we have \[v(\psi_{n\pm 1}(P)) =\frac{(n\pm 1)^2 c_v(P)-c_v((n\pm 1)P)}2= \frac{((n\pm 1)^2-1)c_v(P)}{2}.\] 
		Thus,
		\[k_{v,n}=v(\phi_n(P))=v(\psi_{n- 1}(P)\psi_{n+1}(P))= n^2 c_v(P)=n^2 c_v(P)-c_v(nP).\]
		Here we are using that $c_n(nP)=0$ since $nP$ is non-singular.
		
		It remains to prove the case when $P$ has singular reduction and that  $n_P$ does not exist or it exists but $n_P\nmid n$. In this case $v(x(nP))\geq 0$ and so $k_{v,n}=v(\psi_n^2(P))\geq 0$. Observe that $(P.O)_v=(nP.O)_v=0$ and then from Lemma \ref{lemma:psiinter},
		\[
		k_{v,n}=v(\psi_n^2(P))=n^2c_v(P)-c_v(nP).
		\]
		
		To conclude the proof, we want to show how to compute the values $n^2c_v(P)-c_v(nP)$, as we did in Tables \ref{tab:common_vals} and \ref{tab:common_vals_non_std}. If the Kodaira symbol is not $I_m$, then the value can be computed directly using Lemma \ref{tab:corr_terms}. If the Kodaira symbol is $I_m$, then one can easily prove that $a_{nP}\equiv a_Pn\mod m$. Therefore, \[n^2c_v(P)-c_v(nP)=n^2\dfrac{a_{P}\left( m-a_{P} \right)}{m}-\dfrac{n'\left( m-n' \right)}{m}\]
		with $n'$ the smallest positive integer such that $a_{P}n \equiv n' \bmod{m}$.
		
	\end{proof}
	\section{A corollary and an example}\label{sec:exam}
	In this section, we want to show an application of the main theorem and an example. Let $E$ be an elliptic curve defined by a Weierstrass equation with coefficients in the principal ideal domain $R$ with fraction field $K$. Given a point $P\in E(K)$, one can define a sequence of integral ideals $B_n$ that represents the square root of the denominator of the fractional ideal $(x(nP)) R$ (the fact that the denominator is a square follows from the fact that $E$ is defined by an equation with integer coefficients). The sequence of the $B_n$ is a so-called \textit{elliptic divisibility sequence}.
	If $R$ is a principal ideal domain, we consider a choice of $\beta_n$ such that $B_n=(\beta_n)$ and then we have a sequence $\{\beta_n\}_{n\in \N}$ of elements in $R$ such that $\beta_n$ generates the square root of the denominator of $(x(nP))R$. The choice of $\beta_n$ is clearly not unique.
	In \cite{recurrence}, it is proved that, if $K=\Q$ and we choose the $\beta_n$ in an appropriate way, the sequence $\{\beta_n\}_{n\in \N}$ satisfies a recurrence relation.
	The two main ingredients of the proof of this fact are the study of the behaviour of the sequence $k_{v,n}$ for all finite places in $\Q$ and a recurrence relation that is satisfied by the sequence $\{\psi_n(P)\}_{n\in \N}$.
	Thanks to the main theorem of this paper, we know the behaviour of $k_{v,n}$ in a very general case and then we can generalize the result of \cite{recurrence} to the case when $K$ is not $\Q$. Indeed, we can prove the following.
	\begin{corollary}
		Let $R$ be a principal ideal domain with field of fractions $K$. Let $E$ be an elliptic curve with a minimal Weierstrass model over $R$ and let $P\in E(K)$. Define $M(P)$ as the smallest positive integer such that $[M(P)]P$ is non-singular modulo every maximal ideal of $R$.
		
		We can define a sequence $\{\beta_n\}_{n\in \N}$ of elements in $R$ such that:
		\begin{itemize}
			\item $\beta_n$ generates the square root of the ideal generated by the denominator of $x(nP)$;
			\item For all triples $(n,m,r)$ of positive integers  such that $n\geq m\geq  r$ and any two of them are multiples of $M(P)$, we have
			\[
			\beta_{n+m}\beta_{n-m}\beta_r^2=\beta_{n+r}\beta_{n-r}\beta_m^2-\beta_{m+r}\beta_{m-r}\beta_n^2.
			\]
		\end{itemize}
	\end{corollary}
	\begin{proof}
		The proof follows easily from the proof of \cite[Theorem 1.9]{recurrence} and Theorem \ref{main}, so we just sketch it. We define $\beta_n$ such that $\beta_n^2=\psi_n^2(P)/\gcd(\phi_n(P),\psi_n^2(P))$. We can define the $\gcd$ since $R$ is a PID. We have
		\[
		\psi_{n+m}\psi_{n-m}\psi_r^2=\psi_{n+r}\psi_{n-r}\psi_m^2-\psi_{m+r}\psi_{m-r}\psi_n^2 \quad \text{    for all    }n\geq m\geq r.
		\]
		We divide this equation by \[\gcd(\phi_n(P),\psi_n^2(P))\cdot\gcd(\phi_m(P),\psi_m^2(P))\cdot\gcd(\phi_r(P),\psi_r^2(P))\] and we obtain
		\[
		L_{m,n}\beta_{n+m}\beta_{n-m}\beta_r^2=L_{n,r}\beta_{n+r}\beta_{n-r}\beta_m^2-L_{m,r}\beta_{m+r}\beta_{m-r}\beta_n^2
		\]
		where 
		\[
		L_{m_1,m_2}=\frac{\gcd(\phi_{m_2+m_1}(P),\psi_{m_2+m_1}^2(P))\gcd(\phi_{m_2-m_1}(P),\psi_{m_2-m_1}^2(P))}{\gcd(\phi_{m_2}(P),\psi_{m_2}^2(P))^2\gcd(\phi_{m_2}(P),\psi_{m_2}^2(P))^2}.
		\]
		Note that 
		\[
		v(L_{m_1,m_2})=k_{v,m_1+m_2}+k_{v,m_1-m_2}-2k_{v,m_1}-2k_{v,m_2}.
		\]
		If $m_1$ or $m_2$ is a multiple of $M(P)$, then $v(L_{m_1,m_2})=0$ for each $v$ thanks to Theorem \ref{main}. Therefore, $L_{m_1,m_2}=1$ and, since two of $n,m,r$ are multiples of $M(P)$, we have
		\[
		\beta_{n+m}\beta_{n-m}\beta_r^2=\beta_{n+r}\beta_{n-r}\beta_m^2-\beta_{m+r}\beta_{m-r}\beta_n^2.
		\]
	\end{proof}
	
	Our original goal was to compute $k_{v,n}$ when $K$ is a function field since, as we explained in the introduction, $k_{v,n}$ is related to elliptic divisibility sequences and elliptic divisibility sequences over function fields received much attention in the last years. Hence, we show an example where we explicitly compute $k_{v,n}$ in the function field case. For more details on the following examples, see \cite{Naskrecki_Acta}.
	\begin{example}
		Let $\kappa=\mathbb{C}$, $C=\mathbb{P}_\kappa^1$, and $K=\kappa(C)$. 
		Then, $K=\mathbb{C}(t)$. 
		Consider the elliptic curve $y^2=x(x-f^2)(x-g^2)$ defined over $K$ where $(f,g,h)=(t^2-1,2t,t^2+1)$. Notice that $f^2+g^2=h^2$. Let $P=((f-h)(g-h),(f+g)(f-h)(g-h))\in E(K)$. 
		We denote by $v_1$ the place such that $v_1(t-1)=1$. One can easily show that $E$ is singular modulo $v_1$ since $v_1(f)=1$. Moreover $v_1(g-h)=2$ and then $v_1(x(P))>0$ and $v_1(y(P))>0$. So, $P$ is singular modulo $v$. By direct computation, $v_1(\Delta_E)=4$ and $v_1(j(E))=-4$. Then, $E$ is in minimal form over $K_{v_1}$ and $E/K_{v_1}$ has Kodaira symbol $I_4$.
		By direct computation, $2P=(t^4 + 2t^2 + 1,-2t^5 + 2t )$ and then $2P$ is non-singular modulo $v_1$, observing that $v_1(x(2P))=0$. So, $a_{P,v_1} =2$ and $m_{P,v_1}=2$. Using Lemma \ref{tab:corr_terms}, we have
		\[
		c_{v_1}(P)=\frac{2(4-2)}{4}=1.
		\]
		Therefore, thanks to Theorem \ref{main},
		\[
		k_{v_1,n}(P)=\begin{cases}
			n^2-1	\text{ if $n$ is odd,}\\
			n^2		\text{ if $n$ is even.}
		\end{cases}
		\]
		This agrees with the direct computation of some cases for $n$ small. If $n=2$, then 
		\[
		\psi_2^2(P)=4(x(P)(x(P)-f^2)(x(P)-g^2))=16(t^2+2t-1)^2(t^2-2t+1)^2
		\]
		and $v_1(\psi_2^2(P))=4$.
		If $n=3$, then 
		\[
		\psi_3^2(P)=256(t^4+4t-1)^2(t^2+2t-1)^4(t-1)^8
		\]
		and so $v_1(\psi_3^2(P))=8$. Moreover, $v_1(\psi_4^2(P))=18$ and $v_1(x(P))=2$. Therefore, using $\phi_n(P)=x(P)\psi_n^2(P)-\psi_{n-1}(P)\psi_{n+1}(P)$ we obtain
		\begin{align*}
			\phi_2(P)=&\left(8(t^2+2t-1)(t^2-2t+1)\right)^2(2(t-1)^2)\\&-16(t^4+4t-1)(t^2+2t-1)^2(t-1)^4
		\end{align*}
		and then $v(\phi_2(P))= 4$. In the same way $v_1(\phi_3(P))= 10$. Hence, we conclude that $k_{v_1,2}=4$ and
		$k_{v_1,3}=8$.
	\end{example}
	
		\normalsize
	\baselineskip=17pt
	\bibliographystyle{alpha}
	\bibliography{biblioff}
	
	Bartosz Naskręcki, Faculty of Mathematics and Computer Science, Adam Mickiewicz University in Poznań, ul. Uniwersytetu Poznańskiego 4, 61-614, Poznań, Poland. 
	\\
	EMAIL: \url{bartosz.naskrecki@amu.edu.pl}
	
	Matteo Verzobio, Institute of Science and Technology Austria, Am Campus 1, 3400 Klosterneuburg, Austria. 
	\\
	EMAIL: \url{matteo.verzobio@gmail.com}
	
\end{document}